\documentclass[11pt,a4paper]{article}
\usepackage[latin1]{inputenc}
\usepackage{amsmath}
\usepackage{amsthm}
\usepackage{amsfonts}
\usepackage{amsfonts,amsthm,latexsym,amsmath,amssymb,amscd,epsfig,psfrag,enumerate}
\usepackage{graphics,graphicx, bezier, float, color, hyperref}
\usepackage{amssymb,url}
\usepackage{multienum}
\usepackage[table]{xcolor}
\usepackage{multicol,multirow}
\usepackage{graphicx}
\usepackage{fancyvrb}
\usepackage{tabularx}
\usepackage{array} 
\usepackage{tcolorbox}
\usepackage{parskip}
\usepackage[toc,page]{appendix}
\usepackage[top=2.7cm, bottom=2.7cm, left=1.5cm, right=1.5cm]{geometry}
\usepackage{xcolor}
\usepackage[none]{hyphenat}[section]
\usepackage{blkarray}
\newtheorem{theorem}{Theorem}[section]
\newtheorem{lemma}[theorem]{Lemma}

\newtheorem{definition}[theorem]{Definition}

\numberwithin{equation}{section}
\numberwithin{table}{section}
\numberwithin{figure}{section}

\title{On mixed $b$-concatenations of Fibonacci and Lucas numbers that are Lucas numbers}
\author{Herbert Batte$^{1}$ and Prosper Kaggwa$^{2,*}$ }
\date{}

\begin{document}
	\maketitle
	\abstract{ Let $(F_n)_{n\ge0}$ and $(L_n)_{n\ge0}$ denote the sequences of Fibonacci and Lucas numbers respectively. This paper determines all Lucas numbers that can be represented as base $b$ mixed concatenations of a Fibonacci number and a Lucas number. Mathematically, we study of two Diophantine equations $L_n=b^dL_m+F_k$ and $L_n=b^dF_m+L_k$, where $d$ is the number of digits of $F_k$ or $L_k$ in base $b$. To tackle these equations, we combine tools from Diophantine approximation on non-zero linear forms in logarithms and reduction methods based on continued fractions. This allows us to prove that only finitely many such Lucas numbers exist. } 
	
	
	
	\thanks{$ ^{*} $ Corresponding author}
	
	\section{Introduction}
	\subsection{Background}
	Let $(F_n)_{n\ge0}$ and $(L_n)_{n\ge0}$ be the sequences of Fibonacci and Lucas numbers respectively defined by $$F_0=0, \quad  F_1=1,  \quad F_{n+2}=F_{n+1}+F_n,$$ and $$L_0=2, \quad  L_1=1, \quad L_{n+2}=L_{n+1}+L_n,$$ for $n\ge0.$ These two sequences satisfy the same recurrence relation but have different initial terms. 
	
	Recent work has investigated problems in which a term of a linear recurrence sequence is obtained by concatenating terms within the same sequence. In particular, Banks and Lucas proved in \cite{banks} that if $U_n$ is any non-degenerate binary recurrent sequence of integers, then only finitely many terms of $U_n$ can be expressed as concatenations of two or more of its own terms. As a special case, they showed that $13,21$ and $55$ are the only Fibonacci numbers which are concatenations of two Fibonacci numbers. In \cite{alan}, the authors determined all Fibonacci numbers that can be expressed as  concatenations involving a Fibonacci and a Lucas number. In \cite{AT}, the authors studied Diophantine equations involving Pell and Pell-Lucas numbers arising from base $b$ concatenations of the form 
	\[
	P_n = b^d Q_m + Q_k \quad \text{and} \quad Q_n = b^d P_m + P_k.
	\]
	Similarly in \cite{KM}, the authors considered mixed concatenation-type equations in which Pell numbers are expressed as base $b$ concatenations of Pell and Pell--Lucas numbers. In both papers, all solutions were completely determined for bases $2 \leq b \leq 10$.
	Motivated by the work in \cite{AT} and \cite{KM}, we determine all Lucas numbers which can be written as base $b$ mixed concatenations of a Fibonacci and a Lucas number. For this reason, we study the Diophantine equations
	 \begin{equation}\label{main1}
		L_n=L_m\cdot b^d +F_k
	\end{equation} 
	 and 
	 \begin{equation}\label{main2}
		L_n=F_m\cdot b^d +L_k,
	\end{equation} 
	in non-negative integers $n$, $m$ and $k$, where $d$ denotes the number of digits of $F_k$ and $L_k$ in base $b$ respectively, for $2\leq b \leq 10.$
	
	We prove the following results. 
	\subsection{Main Results}
	\begin{theorem}\label{thm1}
		For $2\le b\le10$, the only solutions to the Diophantine equation \eqref{main1} are;
\begin{table}[H]
	\centering
	\begin{tabularx}{\textwidth}{|c|>{\raggedright\arraybackslash}X|}
			\hline
			\textbf{Base $b$} & \textbf{Solutions $(n,m,k,b,d)$} \\
			\hline
			2 & $(0,1,0,2,1)$, $(2,1,1,2,1)$, $(2,1,2,2,1)$, $(3,0,0,2,1)$, $(4,1,4,2,2)$, $(4,2,1,2,1)$,\newline $(4,2,2,2,1)$, $(5,0,4,2,2)$, $(6,3,3,2,2)$, $(7,1,7,2,4)$, $(7,2,5,2,3)$, $(8,5,4,2,2)$ \\
			\hline
			3 & $(2,1,0,3,1)$, $(3,1,1,3,1)$, $(3,1,2,3,1)$, $(4,0,1,3,1)$, $(4,0,2,3,1)$, $(5,2,3,3,1)$ \\
			\hline
			4 & $(3,1,0,4,1)$, $(4,1,4,4,1)$, $(5,0,4,4,1)$, $(6,3,3,4,1)$, $(7,1,7,4,2)$, $(7,4,1,4,1)$,\newline $(7,4,2,4,1)$, $(8,5,4,4,1)$ \\
			\hline
			5 & $(4,1,3,5,1)$, $(5,0,1,5,1)$, $(5,0,2,5,1)$, $(6,2,4,5,1)$ \\
			\hline
			6 & $(4,1,1,6,1)$, $(4,1,2,6,1)$, $(5,1,5,6,1)$, $(6,2,0,6,1)$, $(7,3,5,6,1)$, $(8,4,5,6,1)$,\newline $(13,0,11,6,3)$ \\
			\hline
			7 & $(4,1,0,7,1)$, $(7,3,1,7,1)$, $(7,3,2,7,1)$ \\
			\hline
			8 & $(5,1,4,8,1)$, $(6,0,3,8,1)$, $(7,2,5,8,1)$ \\
			\hline
			9 & $(5,1,3,9,1)$, $(6,0,0,9,1)$, $(7,2,3,9,1)$ \\
			\hline
			10 & \mbox{$(5,1,1,10,1)$}, $(5,1,2,10,1)$, $(6,1,6,10,1)$ \\
			\hline
		\end{tabularx}
	\end{table}
	
		 In particular, the only Lucas numbers satisfying equation \eqref{main1} when $b=10$ (decimal base) are
		\[
		\begin{aligned}
			L_5 &= 10^1 \cdot L_1 + F_1 =10^1 \cdot L_1 + F_2 = 11,\quad\text{and}\\
			L_6 &= 10^1 \cdot L_1 + F_6 = 18.
		\end{aligned}
		\]				
\end{theorem}
	\begin{theorem}\label{thm2}
		For $2\le b\le10$, the only solutions to the Diophantine equation \eqref{main2} are;
		\begin{table}[H]
			\centering
			\begin{tabularx}{\textwidth}{|c|>{\raggedright\arraybackslash}X|}
				\hline
				\textbf{Base $b$} & \textbf{Solutions $(n,m,k,b,d)$} \\
				\hline
				2 &
				$(2,1,1,2,1)$, $(2,2,1,2,1)$, $(4,1,2,2,2)$, $(4,2,2,2,2)$, $(4,4,1,2,1)$, $(5,3,2,2,2)$,\newline $(5,5,1,2,1)$, $(8,5,4,2,3)$ \\
				\hline
				3 &
				$(3,1,1,3,1)$, $(3,2,1,3,1)$, $(4,3,1,3,1)$, $(5,4,0,3,1)$, $(9,6,3,3,2)$ \\
				\hline
				4 &
				$(4,1,2,4,1)$, $(4,2,2,4,1)$, $(5,3,2,4,1)$ \\
				\hline
				5 &
				$(4,1,0,5,1)$, $(4,2,0,5,1)$, $(5,3,1,5,1)$, $(6,4,2,5,1)$, $(7,5,3,5,1)$ \\
				\hline
				6 &
				$(4,1,1,6,1)$, $(4,2,1,6,1)$, $(8,1,5,6,2)$, $(8,2,5,6,2)$ \\
				\hline
				7 &
				$(5,1,3,7,1)$, $(5,2,3,7,1)$, $(6,3,3,7,1)$ \\
				\hline
				8 &
				$(5,1,2,8,1)$, $(5,2,2,8,1)$, $(6,3,0,8,1)$, $(8,5,4,8,1)$, $(14,7,5,8,2)$ \\
				\hline
				9 &
				$(5,1,0,9,1)$, $(5,2,0,9,1)$, $(7,4,0,9,1)$, $(8,5,0,9,1)$, $(9,6,3,9,1)$ \\
				\hline
				10 &
				$(5,1,1,10,1)$, $(5,2,1,10,1)$ \\
				\hline
			\end{tabularx}
		\end{table}
		
		 In particular, the only Lucas number satisfying \eqref{main2} when $b=10$ (decimal base) is
		\[
		\begin{aligned}
			L_5 &= 10^1 \cdot F_1 + L_1 = 10^1 \cdot F_2 + L_1 = 11.
		\end{aligned}
		\] 
	\end{theorem}
	
	\medskip
	
	 To establish these results, we apply two powerful tools from Diophantine approximation. The first is Matveev's theorem from \cite{matl} on lower bounds for non-zero linear forms in logarithms of algebraic numbers, which provide explicit upper bounds for the variables. The second is a reduction method based on a theory of continued fractions, derived from the Baker-Davenport Lemma and refined by Dujella and Peth\H{o}, \cite{duj}. Together, these techniques allow us to prove that the equations \eqref{main1} and \eqref{main2} have only the solutions stated in Theorems \ref{thm1} and \ref{thm2} respectively. All necessary computations were performed with the aid of SageMath, with precision up to $1000$ digits. 
	 
	\section{Methods}
	\subsection{Some results on Fibonacci and Lucas numbers}
	Let $F_n$ and $L_n$ be the sequence of Fibonacci and Lucas numbers respectively. They both share the same characteristic polynomial $ x^2-x-1=0,$ whose roots are 
	\[\alpha= \dfrac{1+\sqrt{5}}{2} \quad \text{and } \qquad \beta=\dfrac{1-\sqrt{5}}{2}.\] 
	Furthermore, the roots $\alpha $ and $\beta$ are related by $\beta=-\alpha^{-1}.$ The Binet formulae for Fibonacci and Lucas numbers are given by 
	\[F_n=\dfrac{\alpha^n-\beta^n}{\sqrt{5}} \qquad \text{and } \qquad L_n=\alpha^n+\beta^n,\] 
	respectively for $n\geq0.$ Moreover, it was shown in \cite{alan} relates the $n^{\text{th}}$ Fibonacci number is bounded by powers of $\alpha$ via the inequalities	
	\begin{align*}
		\alpha^{n-2}\leq F_n\leq \alpha^{n-1},
	\end{align*}
	for all $n\ge 1$.
	 Similarly, Bravo and Luca showed in \cite{BravoLuca2014} that 
	 \begin{align*}
		\alpha^{n-1}\leq L_n\leq 2\alpha^{n},
	\end{align*} 
	holds for all $n\ge 0$.
	
		\subsection{Bounds on variables from Equation \eqref{main1}}\label{prel1}
	
	Here, we return to \eqref{main1} and establish relationships between the variables contained in there. The number of digits of $F_k$ in base $b$ can be written as 
	$$d=\left \lfloor\log_b {F_k} \right \rfloor +1,$$
	 where $\left \lfloor\Phi\right \rfloor$ is the greatest integer less than or equal to $\Phi$, known as the floor function. Thus,
	  \begin{equation*}
		d=	\left\lfloor\log_b {F_k} \right\rfloor +1 \leq \log_b {F_k} + 1 \leq 1 + \log_b {\alpha^{k-1}} = 1 + (k-1) \log_b \alpha <k, 
		\end{equation*}   
		where we have used the fact that $\log_b \alpha <1$, for $2 \le b \le 10$.
	Also, 
	\begin{equation*}
		d=\left\lfloor\log_b {F_k} \right\rfloor +1 > \log_b {F_k} > \log_b {\alpha^{k-2}} = (k-2) \log_b \alpha >\frac{k-2}{5}, 
		\end{equation*} 
		for which we have used the fact that $\log_b \alpha>1/5$, for $2\le b\le 10.$ Thus we get 
		\begin{equation*}
		\frac{k-2}{5} <d < k.
	\end{equation*}
	On the other hand, \begin{equation}\label{6}
		F_k \leq b^{\left \lfloor\log_b {F_k} \right \rfloor} < b^d =b^{1+\left \lfloor\log_b {F_k} \right \rfloor}=bF_k.
	\end{equation}
	From inequalities \eqref{6} and \eqref{main1}, we have that 
	\begin{align}\label{pp}
		\alpha^{n-1}&<L_n=L_m\cdot b^d +F_k \nonumber\\
		&< bL_mF_k+bF_k =bF_k (L_m+1)\nonumber\\
		&\leq b\alpha^{k-1}(2\alpha^m+1)\leq 3b\alpha^{k+m-1},
	\end{align}   
	and 
	\begin{align}\label{pq}
		2\alpha^n >L_n= L_m\cdot b^d +F_k >F_kL_m +F_k \geq F_kL_m>\alpha^{k+m-3}.
	\end{align} 
	The inequalities \eqref{pp} and \eqref{pq} yield \begin{align*}
		k+m-5 &<n< k+m + \frac{\log {3b}}{\log \alpha} \\&< k+m+3+\frac{\log b}{\log \alpha}\\&< k+m+7\log {b},\end{align*}
	or equivalently,
	 \begin{equation}\label{7}
		k+m-5<n<k+m+7\log {b}.
	\end{equation}	
	Furthermore, we can deduce from the series of inequalities 
	\begin{align*}
		2^d<b^d \leq b^d +\frac{F_k}{L_m} = \frac{L_n}{L_m} \leq \frac{2\alpha^n}{\alpha^{m-1}}=2\alpha^{n-m+1} < 2^{n-m+2},
	\end{align*}   
	 that 
	 \begin{equation}\label{b}
		d<n-m+2.
	\end{equation}
	We use these inequalities more often in the proof of Theorem \ref{thm1}.
	
\subsection{Bounds on variables from Equation \eqref{main2}}\label{prel2}	
Next, suppose Equation \eqref{main2} holds. We also establish some inequalities that will be useful while proving Theorem \ref{thm2}. Here,
 \begin{align*}
	d&=	\left\lfloor\log_b {L_k} \right\rfloor +1 \leq \log_b {L_k} + 1 \leq 1 + \log_b {2\alpha^{k}} \\&= 1 +\log_b 2+ k \log_b \alpha = 1+\frac{\log 2}{\log b}+\frac{k\log \alpha}{\log b}\\&<k+2, 
	\end{align*}   for $2\leq b\leq 10,$ and 
	 \begin{align*}
	d&=\left\lfloor\log_b {L_k} \right\rfloor +1 > \log_b {L_k} \\&\geq \log_b {\alpha^{k-1}} = (k-1) \log_b \alpha \\&>\frac{k-1}{2}. 
	\end{align*}   
	Therefore
	  \begin{equation*}
	\frac{k-1}{2} <d < k+2.
\end{equation*}
Further, we have \begin{equation*}
	L_k \leq b^{\left \lfloor\log_b {L_k} \right \rfloor} < b^d <b^{1+\left \lfloor\log_b {L_k} \right \rfloor}\leq b^{1+\log_b {L_k} }= b\cdot b^{\log_b {L_k} }=bL_k,
\end{equation*}   which yields 
\begin{equation}\label{s}
	L_k<b^d\leq bL_k.
\end{equation}
By \eqref{s}, we see that 
\begin{align*}	
	\alpha^{n-1}&<L_n=F_m\cdot b^d +L_k < bL_kF_m+L_k\\& \leq (b+1)L_kF_m \leq 2(b+1)\alpha^{m+k-1}\\&=\alpha^{\log_\alpha {2(b+1)}}\cdot \alpha^{m+k-1},
\end{align*}   
and thus 
\begin{align}\label{t}
	\alpha^{n-1}<\alpha^{m+k-1+\log_\alpha {2(b+1)}}.
\end{align}  
Also, we have 
\begin{align*}
	2\alpha^n &>L_n= F_m\cdot b^d +L_k \\
	&>L_kF_m +L_k >L_kF_m\\
	&>\alpha^{m+k-3},
\end{align*}   
which implies that
$
	2\alpha^n=\alpha^{n+\log_\alpha 2} >\alpha^{m+k-3}.
$
Therefore,
\[
\alpha^{m+k-3}<\alpha^{n+\log_\alpha 2}<\alpha^{n+2},
\]
which implies
\begin{equation}\label{adu}
	\alpha^{m+k-6}<\alpha^{n-1}.
\end{equation}
Combining \eqref{adu} with \eqref{t}, we get
\[
\alpha^{m+k-6}<\alpha^{n-1}<\alpha^{m+k-1+\log_\alpha2(b+1)},
\]
and therefore
\[
m+k-5<n<m+k+\log_\alpha2(b+1).
\]
Furthermore,
\begin{align*}
	\log_\alpha2(b+1)
	&= \frac{\log2(b+1)}{\log\alpha}
	= \frac{\log\left(2b\left(1+\frac{1}{b}\right)\right)}{\log\alpha}
	< \frac{\log(4b)}{\log\alpha} \\
	&= \frac{\log 4 + \log b}{\log\alpha}
	= \frac{(\log_b 4+1)\log b}{\log\alpha} 
	< 7\log b .
\end{align*}
Hence,
\begin{equation}\label{T2}
	m+k-5<n<m+k+7\log b .
\end{equation}
If $m=0$, then $F_0=0$ and thus equation \eqref{main2} becomes $L_n=L_k$ and thus all Lucas numbers such that $n=k$ satisfy the Diophantine equation \eqref{main2}. Therefore we consider the case $m>0$. Since $m>0$, then we can deduce from \eqref{s} that \begin{align*}
	L_n=b^dF_m +L_k\geq b^d+L_k>2L_k
\end{align*}  
and therefore 
\begin{align*}
	n-k\geq 1.
\end{align*} 
Finally here, from 
\begin{equation}\label{eq:d}
	b^d< b^d + \dfrac{L_k}{F_m}=\dfrac{L_n}{F_m}\leq \dfrac{2\alpha^n}{\alpha^{m-2}}=2\alpha^{n-m+2}
	\end{equation} 
	we have that $d<n-m+3.$	
We shall also use these inequalities later in the proof of Theorem \ref{thm2}.

	\subsection{Linear Forms in Logarithms}
	We use Baker-type lower bounds for nonzero linear forms in logarithms of algebraic numbers. There are many such bounds mentioned in the literature like that of Baker and W\"ustholz from \cite{BW} or Matveev from \cite{matl}. Before we can formulate such inequalities, we need the notion of height of an algebraic number recalled below.
	
	\begin{definition}[Logarithmic height]
		Let $\gamma$ be an algebraic number of degree $d$ with minimal primitive polynomial over the integers, given by
		\[
		a_0 x^d + a_1 x^{d-1} + \cdots + a_d = a_0 \prod_{i=1}^d (x - \gamma^{(i)}),
		\]
		where the leading coefficient $a_0$ is positive. Then the logarithmic height of $\gamma$ is given by 
		\begin{equation*}
			h(\gamma) := \frac{1}{d} \left( \log a_0 + \sum_{i=1}^d \log \max \{ |\gamma^{(i)}|, 1 \} \right). \label{eq:log_height}
		\end{equation*}
	\end{definition} 
	  In particular, if $\gamma$ is a rational number represented as $\gamma := p/q$ with coprime integers $p$ and $q \ge 1$, then 
	\[
	h(\gamma) = \log \max \{ |p|, q \}.
	\]
	The following properties of the logarithmic height function $h(\cdot)$ will be used in the rest of the paper without further reference:
	\begin{align*}
		h(\gamma_1 \pm \gamma_2) &\le h(\gamma_1) + h(\gamma_2) + \log 2,\\
		h(\gamma_1 \gamma_2^{\pm 1}) &\le h(\gamma_1) + h(\gamma_2),\\
		h(\gamma^s) &= |s| h(\gamma) \quad \text{valid for } s \in \mathbb{Z}.	
	\end{align*}
	
	 A linear form in logarithms is an expression
	\begin{align*}
		\Lambda := b_1 \log \gamma_1 + \cdots + b_t \log \gamma_t ,
	\end{align*}
	where for us $\gamma_1, \ldots, \gamma_t$ are positive real algebraic numbers and $b_1, \ldots, b_t$ are non--zero integers. We assume $\Lambda \neq 0$. We need lower bounds for $|\Lambda|$. We write $K := \mathbb{Q}(\gamma_1, \ldots, \gamma_t)$ and $D$ for the degree of $K$. We start with a modified version of Matveev's result due to Bueaud, Mignotte and Siksek, \cite{BugeaudMignotteSiksek2006}.
	
	\begin{theorem}[Matveev, \cite{BugeaudMignotteSiksek2006}]\label{log}
		Let $\Gamma := \gamma_1^{b_1} \cdots \gamma_t^{b_t} - 1 = e^\Gamma - 1$ and assume $\Gamma \neq 0$. Then 
		\begin{equation*}
			\log |\Gamma| > -1.4 \cdot 30^{t+3} \cdot t^{4.5} \cdot D^2 (1 + \log D) (1 + \log B) A_1 \cdots A_t, \label{eq:matveev}
		\end{equation*}
		where $B \ge \max \{|b_1|, \ldots, |b_t|\}$ and $A_i \ge \max \{ D h(\gamma_i), |\log \gamma_i|, 0.16 \}$ for all $i = 1, \ldots, t$.
	\end{theorem}

	\subsection{Reduction Methods}
	
	Typically, the estimates from Theorem \ref{log} are excessively large to be practical in computations. To refine these estimates, we employ a modified approach based on the Baker-Davenport reduction method. Our adaptations follow the method introduced by Dujella and Peth\H{o} (\cite{duj}, Lemma 5). When considering a real number $r$, we use $\| r \|$ to represent the smallest distance between $r$ and any nearest integer which is formally written as $\min \{ |r - v| : v \in \mathbb{Z} \}$.
	\begin{lemma}[Dujella \& Peth\H{o}
		, \cite{duj}]\label{duj}
		Let $\tau $ be irrational and $A, B, \mu$ be real numbers with $A > 0$ and $B > 1$. Let $M > 1$ be a positive integer and suppose that $p/q$ is a convergent of the continued fraction expansion of $\tau$ with $q > 6M$. Let 
		\[
		\varepsilon := ||\mu q|| - M \|| \tau q \||.
		\]
		If $\varepsilon > 0$, then there is no solution of the inequality 
		\begin{equation*}
			0 <| u \tau - v + \mu |< A B^{-w},
		\end{equation*}
		in positive integers $u, v, w$ with 
		\begin{equation*}
			\frac{\log (Aq/\varepsilon)}{\log{B}} \ge w \quad \text{and} \quad m < M.
		\end{equation*}
	\end{lemma}
	  Lemma \eqref{duj} cannot be applied where $\mu=0$ or when $\mu$ is a linear combination of $1$ and $\gamma$ with integer coefficients, as in such cases $\varepsilon$ becomes negative for sufficiently large $M$. When this happens, a classical result due to Legendre (see Theorem 8.2.4 in \cite{qu}) is applied instead.
	\begin{lemma}[Legendre, \cite{qu}]\label{legendre}
		Let $\tau$ be an irrational number, $M$ be a positive integer and $\frac{p_k}{q_k}(k=0,1,2,\cdots)$ be all the convergents of the continued fraction $[a_0,a_1,\cdots]$ of $\tau.$ let $N$ be such that $q_N>M.$ Then putting $a_M:=\max\{a_i:~i=0,1,\cdots ,N\},$ the inequality \begin{align*}
			|u\tau-v|>\frac{1}{(a_M+2)u},
		\end{align*} holds for all pairs $(v,u) $of integers with $0<u<M.$   
	\end{lemma}
	  Next, we prove the following fact from calculus which will be used later throughout the paper.
	\begin{lemma}\label{Analysis}
		If $x \in \mathbb{R} $ satisfies $|x|<\frac{1}{2}$, then \[\frac{1}{2}|x|<|\log{(1+x)}|<\frac{3}{2}|x|.\]
		\begin{proof}
		If $x \in \mathbb{R} $ satisfies $|x|<\frac{1}{2}$, then we obtain the upper bound on $|\log (1+x)|$ as follows: \begin{align*}|\log{(1+x)}|&=\left|x-\frac{x^2}{2}+\frac{x^3}{3}-\frac{x^4}{4}+\cdots \right|,\\&<|x|+\frac{1}{2}\left(|x|^2+|x|^3+|x|^4+\cdots \right)=|x|\left(1+\frac{|x|}{2(1-|x|)}\right)<\frac{3}{2}|x|.\end{align*}
		Similarly, if $x \in \mathbb{R} $ satisfies $|x|<\frac{1}{2}$, then we obtain the lower bound on $|\log {(1+x)}|$ as follows: \begin{align*}
			|x-\log{(1+x)}|&=\left|\frac{x^2}{2}-\frac{x^3}{3}+\frac{x^4}{4}-\cdots\right|,\\&<\frac{|x|}{2}\left(|x|+|x|^2+|x|^3+\cdots\right)=|x|\left(\frac{|x|}{2(1-|x|)}\right)<\frac{1}{2}|x|.
		\end{align*}
		The desired inequality imediately follows by noting that $|x|\leq |x-\log{(1+x)}|+|\log{(1+x)}|.$ 
		That is, \[|\log{(1+x)}\geq |x|-|x-\log{(1+x)}|>|x|-\frac{1}{2}|x|=\frac{1}{2}|x|.\]
		\end{proof}
	\end{lemma}
	  Finally, we present an analytical argument which is Lemma 7 in \cite{guz}.
	\begin{lemma}[Lemma 7 in \cite{guz}]\label{reduction}
		If $s \ge 1$, $T > (4s^2)^s$ and $T > Z/(\log{Z})^s$, then 
		\begin{equation*}
			Z < 2^sT( \log T)^s.
		\end{equation*}
	\end{lemma}
	  SageMath was used to perform all the computations in this work.
	
	\section{Proof of the Main Results}
	In Subsections \ref{prel1} and \ref{prel2}, we established some relationships between the variables in \eqref{main1} and \eqref{main2} respectively. We use these relationships in the proof of the main results.

	\subsection{Proof of Theorem \ref{thm1}}
	When $n\leq1500$, we have that $m+k <n+5$ via \eqref{7} and $d<n-m+2$ via \eqref{b}. A quick computer search using  SageMath yields the results stated in Theorem \ref{thm1} as the only solutions. So from now on, we assume in this subsection that $n>1500.$
	
	We start by proving the following lemma. 
	\begin{lemma}\label{lem3.1}
		All solutions to the Diophantine equation \eqref{main1} satisfy 
		\[n-k<2.3\cdot 10^{10}\log b \log {(n-m+2)}.\]
	\end{lemma} 
	\begin{proof}  To prove this, we first rewrite \eqref{main1} as 
		\begin{align*} 
			\alpha^n +\beta^n &=(\alpha^m +\beta^m)b^d+F_k, 
			\\
			 \alpha^n-\alpha^mb^d &= \beta^mb^d-\beta^n +F_k.
		\end{align*}   
		Dividing through by $\alpha^n$ and taking absolutes on both sides, we get 
		\begin{align*}
			\left|1-\frac{b^d}{\alpha^{n-m}} \right| \leq \frac{\left|\beta^m \right|}{\alpha^n}b^d+\frac{\left|\beta^n \right|}{\alpha^n}+\frac{F_k}{\alpha^n}.
		\end{align*}   
		Note that $|\beta|<1, F_k<\alpha^{k-1}$ and so we can write 
		\begin{align*}
			\left|1-\frac{b^d}{\alpha^{n-m}}\right|&\leq \frac{b^d}{\alpha^n}+\frac{1}{\alpha^n}+\frac{\alpha^{k-1}}{\alpha^n} 
			=\frac{1+b^d+\alpha^{k-1}}{\alpha^n} \\
			&< \frac{\alpha^k+bF_k+\alpha^k}{\alpha^n}
			< \frac{2\alpha^k+b\alpha^{k-1}}{\alpha^n}\\ 
			&<\frac{(b+2)\alpha^k}{\alpha^n} 
			 \leq \frac{2b}{\alpha^{n-k}},
		\end{align*}   
		which holds for all $n>1500$. 
		  Fix \begin{align}\label{a}
			\Gamma_1:=\left|1-\frac{b^d}{\alpha^{n-m}} \right|<\frac{2b}{\alpha^{n-k}}.
		\end{align}
		  Clearly, $\Gamma_1\neq 0$ otherwise we would have $\alpha^{n-m}=b^d$, which is a contradiction since $\alpha^{n-m}$ is a unit in $\mathbb{Q}(\sqrt{5})$ while $b^d $ is not. Take 
		\begin{align*}
			\gamma_1 &:= b,      & b_1 &:= d,\\
			\gamma_2 &:= \alpha,& b_2 &:= -(n-m).
		\end{align*}
		 Moreover, since $\gamma_1$ and $\gamma_2$ belong to $\mathbb{Q}(\sqrt{5})$, we take $D:=2$.
		Note that
		\begin{align*}
			h(\gamma_1) = h(b) = \log b, \qquad
			h(\gamma_2) = \tfrac{1}{2}\log \alpha.
		\end{align*}		 
		Hence, we obtain $A_1 := 2\log b$ and $A_2 := \log \alpha$, with $t=2$.
		From \eqref{b}, we have
		\[
		B := \max\{d,\, n-m\} = n-m+2 > d.
		\]
Using Theorem \ref{log}, we get  \begin{align*}
			\log {|\Gamma_1|} &> -1.4\cdot 30^5\cdot 2^{4.5}\cdot 2^2 (1+\log 2)(1+\log{(n-m+2)})\cdot 2\log b \cdot \log \alpha, \\ &>-1.1 \cdot10^{10}\log {(n-m+2)}\log b, 
			\end{align*}   
			  Comparing this bound on $\log |\Gamma_1|$ with \eqref{a}, we get
			  \begin{equation}\label{p}
			n-k<2.3\cdot 10^{10}\log b \log {(n-m+2)},
	\end{equation}
 and the proof of Lemma \ref{lem3.1} is complete. 
 \end{proof}   

Next, we prove the following. 
\begin{lemma}\label{lem3.2}
		All solutions to the Diophantine equation \eqref{main1} satisfy \[n<4.4\cdot 10^{27}(\log b)^4. \]
	\end{lemma} \begin{proof}
		  Again, we rewrite equation \eqref{main1} as \begin{align*}
			\alpha^n +\beta^n &= L_m\cdot b^d + \frac{\alpha^k -\beta ^k}{\sqrt{5}}, \\ \alpha^n-\frac{\alpha^k}{\sqrt{5}}-L_m \cdot b^d&=\frac{-\beta^k}{\sqrt{5}}-\beta^n, \\ \alpha^n \left(1-\frac{\alpha^{k-n}}{\sqrt{5}}\right)-L_m\cdot b^d &=\frac{-\beta^k}{\sqrt{5}}-\beta^n.
		\end{align*}
		  Taking absolutes on both sides and dividing through by $\alpha^n \left(1-\alpha^{k-n}/\sqrt{5}\right)$ gives  
		\begin{align*}
		\left| 1-\frac{L_m\cdot b^d}{\alpha^n \left(1-\dfrac{\alpha^{k-n}}{\sqrt{5}}\right)}\right| &\leq \frac{\frac{|\beta^k|}{\sqrt{5}}+|\beta^n|}{\alpha^n \left(1-\dfrac{\alpha^{k-n}}{\sqrt{5}}\right)}. \end{align*}  
		Notice that $\beta=-\alpha^{-1}$, so $|\beta|<1$ since $\alpha>1$.
		  We therefore have 
		\begin{align*}
		\left| 1-\frac{L_m\cdot b^d}{\alpha^n \left(1-\dfrac{\alpha^{k-n}}{\sqrt{5}}\right)}\right|&<\frac{2}{\alpha^n}\cdot \frac{1}{\left(1-\dfrac{\alpha^{-(n-k)}}{\sqrt{5}}\right)}.
		\end{align*}   
		The quantity 
		\begin{align*}
		\frac{1}{\left(1-\dfrac{\alpha^{-(n-k)}}{\sqrt{5}}\right)}&=\frac{\sqrt{5}}{\sqrt{5}-\alpha^{-(n-k)}}<\frac{\sqrt{5}}{\sqrt{5}-\alpha^{-0}}=\frac{\sqrt{5}}{\sqrt{5}-1}<2,
		\end{align*}  
		so we have 
		\begin{align}\label{gamma2}
			\Gamma_2:= \left| 1-\frac{L_m\cdot b^d}{\alpha^n \left(1-\dfrac{\alpha^{k-n}}{\sqrt{5}}\right)}\right| < \frac{4}{\alpha^n}.
		\end{align}   
		We note that $\Gamma_2\neq 0$, otherwise we would have 
		$$\alpha^n-\frac{\alpha^k}{\sqrt{5}}=L_m\cdot b^d,$$ 
		which is not true because the left hand side is irrational for $n>1500$ while the right hand side is an integer. Take the field $\mathbb{Q}(\sqrt{5})$ of degree $D=2$ and $t=3$.
		\begin{align*}
			\gamma_1 &:= b, 
			& \gamma_2 &:= \alpha,
			& \gamma_3 &:= \dfrac{L_m}{\left(1-\dfrac{\alpha^{k-n}}{\sqrt{5}}\right)},\\[0.3em]
			b_1 &:= d,
			& b_2 &:= -n,
			& b_3 &:= 1.
		\end{align*}
		As before, we take $A_1 := 2\log b$ and $A_2 := \log \alpha$.		 
		Observe that
		\begin{align*}
			h(\gamma_3)
			&\le h(L_m)+h\!\left(1-\dfrac{\alpha^{k-n}}{\sqrt{5}}\right), \\
			&< h(2\alpha^m)+h(1)+|n-k|\,h(\alpha)+h(\sqrt{5})+\log 2, \\
			&\le \log 2 + \tfrac{m}{2}\log \alpha + \tfrac{n-k}{2}\log \alpha
			+ \tfrac{1}{2}\log 5 + \log 2.
		\end{align*}
		From \eqref{7}, \(n-k+5>m\), and thus
		\begin{align*}
			h(\gamma_3)
			< 2\log 2 + \tfrac{1}{2}\log 5 + \tfrac{2(n-k)+5}{2}\log \alpha.
		\end{align*}
		Therefore,
		\begin{align*}
			A_3
			&= 4\log 2 + \log 5 + (2(n-k)+5)\log \alpha, \\
			&< 4\log 2 + \log 5
			+ \bigl(2(2.3\cdot 10^{10}\log b \log (n-m+2))+5\bigr)\log \alpha, \\
			&< 2.3\cdot 10^{10}\log b \log (n-m+2).
		\end{align*}		 
		Taking $B := n-m+2$, we proceed via Theorem \ref{log} and obtain
		\begin{align*}
			\log {|\Gamma_2|}
			&> -1.4\cdot 30^6 \cdot 3^{4.5}\cdot 2^2 (1+\log 2)(1+\log (n+2)) \cdot 2\log b \cdot \log \alpha
			\cdot 2.3\cdot 10^{10} \log b \log (n-m+2), \\
			&> -8.1\cdot 10^{22} (\log n)^2 (\log b)^2 ,
		\end{align*}		 
		where we have used the facts that
		\begin{align*}
			\log (n-m+2)&< \log (n+2)< 3\log n,
		\end{align*}
		and
		\begin{align*}
		1+\log (n+2)
		&< 1+3\log n< 5\log n.
		\end{align*}
		Comparing the bounds for \(\log |\Gamma_2|\) with \eqref{gamma2}, we obtain
		\begin{align}\label{guz}
			n < 1.7\cdot 10^{23} (\log n)^2 (\log b)^2.
		\end{align}
Next, we apply Lemma \ref{reduction} to \eqref{guz} with parameters $Z=n, s=2$ and $T=1.7\cdot 10^{23}(\log b)^2$. We obtain 
 \begin{align}\label{q}
	n<4.4\cdot 10^{27}(\log {b})^4,
	\end{align}
and the proof of Lemma \ref{lem3.2} is complete.	
\end{proof}

	  We proceed to reduce the bounds on $n.$  Recall that 
	  \begin{align*}
		\Gamma_1=\left|1-\frac{b^d}{\alpha^{n-m}} \right|= \left|e^{\Lambda_1}-1\right|. 
	\end{align*}   
	We already showed that $\Gamma_1\neq 0$ and thus $\Lambda_1\neq 0$. Since $|\Gamma_1|<1/2$ whenever $n-k>8$, then \[|\log{(1+\Gamma_1)}|<\frac{3}{2}|\Gamma_1|<\frac{3}{2}\left(\frac{2b}{\alpha^{(n-k)}}\right)=\frac{3b}{\alpha^{n-k}}.\] 
	Therefore
	 \begin{align*}
		\Lambda_1:=|d\log b -(n-m)\log \alpha|< \frac{3b}{\alpha^{n-k}},
	\end{align*}   
	and so we can write
	 \begin{align*}
		0<\left|\frac{\log b}{\log \alpha} -\frac{n-m}{d}\right|< \frac{3b}{\alpha^{n-k}\cdot d\log \alpha}.
	\end{align*}   
	Using the fact that $b\leq 10$ and combining \eqref{q} with \eqref{b}, we obtain an upper bound for $d$ as \begin{align*}
		d<n-m+2<n+2<1.3\cdot 10^{29}.
	\end{align*}  
	Proceeding via Lemma \ref{legendre}, we have that $(n-m)/d$ is a convergent of the continued fraction expansion of $\log b/\log \alpha$ for $2\leq b\leq 10$. Define
	 \[
	\tau := \dfrac{\log b}{\log \alpha}, \qquad
	u := d, \qquad
	M := 1.3 \cdot 10^{29}.
	\] 
	Let $q_i$ be the denominator of the $i$-th convergent of the continued fraction of $\tau$. By applying a property of continued fractions, we write \begin{align*}
		\frac{1}{(a_{max}+2)d^2}\leq \frac{1}{(a_i+2)d^2}\leq \left|\frac{\log b}{\log \alpha} -\frac{n-m}{d}\right|< \frac{7b}{d\cdot \alpha^{n-k}}.
	\end{align*}    A simple computer program in SageMath yields the following results.
	
	\begin{center}
		\begin{tabular}{|c|c|c|c|c|c|c|c|c|c|}
			\hline
			$b$ & $2$ & $3$ & $4$ & $5$ & $6$ & $7$ & $8$ & $9$ & $10$ \\ \hline
			$q_i$ & $q_{67}$ & $q_{62}$ & $q_{64}$ & $q_{59}$ & $q_{57}$ & $q_{67}$ & $q_{59}$ & $q_{58}$ & $q_{61}$ \\ \hline
			$q_i>$ & $2\cdot 10^{29}$ & $1.2\cdot10^{30}$ & $1.6\cdot10^{29}$ & $2.9\cdot10^{29}$ & $7.1\cdot 10^{29}$ & $3\cdot 10^{29}$ & $2.4\cdot 10^{29}$ & $1.1\cdot 10^{30}$ & $2.6\cdot 10^{30}$ \\ \hline
			$a(M)$ & $a_{18}$ & $a_{43}$ & $a_{18}$ & $a_{21}$ & $a_{47}$ & $a_{14}$ & $a_{16}$ & $a_{39}$ & $a_{36}$ \\ \hline
			$a(M)$ & $134$ & $161$ & $66$ & $59$ & $347$ & $35$ & $44$ & $80$ & $106$ \\ \hline
		\end{tabular}
	\end{center}	
	  In all these cases, we have
	\begin{align*}
		n-k < \frac{\log {(7b(347+2)\cdot 1.3\cdot 10^{29})}}{\log \alpha} < 161,
	\end{align*}
	  which implies that
	$	m<n-k+5 < 166$.
	
	  Next, we go back to \eqref{gamma2} and write
	\begin{align*}
		\Lambda_2:= d\log b -n\log \alpha + \log{\left(\frac{L_m}{1-\frac{\alpha^{k-n}}{\sqrt{5}}}\right)}.
	\end{align*}
  Assume for a moment that $n \geq 3$ so that $|\Gamma_2| < 0.5$. Then
\begin{align*}
	\Lambda_2 	= \left| \log(1+\Gamma_2) \right| < \frac{3}{2} |\Gamma_2| < \frac{3}{2} \left(\frac{4}{\alpha^n}\right) = \frac{6}{\alpha^n},
\end{align*}
and so,
	\begin{align*}
		0<\left|d\frac{\log b}{\log \alpha}-n+\frac{\log{\left(\frac{L_m}{1-\frac{\alpha^{k-n}}{\sqrt{5}}}\right)}}{\log \alpha}\right|<\frac{6/\log \alpha}{\alpha^n}.
	\end{align*}
	  We can now apply Lemma \ref{duj} with parameters  \[
	w := n, \qquad
	A := \dfrac{6}{\log \alpha}, \qquad
	B := \alpha , 
	\] $M:=1.3\cdot 10^{39}>n+2>n-m+2>d$, and 
	 \[\tau:=\frac{\log b}{\log \alpha},\qquad \mu:=\frac{\log{\left(\frac{L_m}{1-\frac{\alpha^{k-n}}{\sqrt{5}}}\right)}}{\log \alpha}. \]
	Lets denote with $q_i$ the denominator of the $i$-th convergent of the continued fraction of $\tau.$ 
	We implement the algorithm of Lemma \ref{duj} in SageMath and for each $2\leq b\leq 10$, $1\leq m\leq 166$ and $3\leq n-k < 161,$ we got the largest numerical bounds for $w:=n$ for the following data.
	\begin{center}
		\begin{tabular}{|c|c|c|c|c|c|c|c|c|c|}
			\hline
			$b$ & $2$ & $3$ & $4$ & $5$ & $6$ & $7$ & $8$ & $9$ & $10$ \\ \hline
			$m$ & $87$ & $3$ & $87$ & $10$ & $52$ & $69$ & $39$ & $107$ & $2$ \\ \hline
			$n-k>$ & $95$ & $38$ & $95$ & $111$ & $155$ & $66$ & $140$ & $126$ & $66$ \\ \hline
			$q_i$ & $q_{87}$ & $q_{74}$ & $q_{86}$ & $q_{83}$ & $q_{81}$ & $q_{83}$ & $q_{78}$ & $q_{72}$ & $q_{87}$ \\ \hline
			$\varepsilon>$ & $0.0001$ & $0.00002$ & $0.0001$ & $0.00001$ & $2\cdot 10^{-7}$ & $0.00002$ & $0.00001$ & $0.00003$ & $0.0004$ \\ \hline
			$n<$ & $223$ & $221$ & $223$ & $223$ & $229$ & $219$ & $221$ & $219$ & $220$ \\ \hline 
		\end{tabular}
	\end{center}
	
	  Notice that in all cases, we have $n\leq 229$ holds. This contradicts our working assumption that $n> 1500$, and the proof of Theorem \ref{thm1} is done.
	  
	\subsection{Proof of Theorem \ref{thm2}} 
	
	Note that when $n\leq 1500$, we have that $m+k <n+5$ via \eqref{T2} and $d<n-m+3$ via \eqref{eq:d}. A quick computation using SageMath yields only the solutions stated in Theorem \ref{thm2}. Hence from now on we may assume $n>1500$.
	
	 We proceed to examine Equation \eqref{main2} in two different ways. First, we prove the following Lemma.
	
	\begin{lemma}\label{lem3.3}
		All solutions to the Diophantine equation \eqref{main2} satisfy \begin{align*}
			n-k<8.0\cdot 10^{12}\log b (1+\log {(n-m+3)}).
		\end{align*}
	\end{lemma}
	\begin{proof}
		We rewrite equation \eqref{main2} as \begin{align*}
			\alpha^n +\beta^n &= b^d \left(\frac{\alpha^m-\beta^m}{\sqrt{5}}\right)+L_k \\ \alpha^n +\beta^n &= \frac{\alpha^mb^d}{\sqrt{5}}-\frac{\beta^mb^d}{\sqrt{5}} +L_k \\ \alpha^n-\frac{\alpha^mb^d}{\sqrt{5}}&=L_k-\frac{\beta^mb^d}{\sqrt{5}}-\beta^n.
		\end{align*}  Dividing through by $\alpha^n$ and taking absolutes on both sides yields \begin{align*}
			\left|1-\frac{b^d}{\alpha^{n-m}\sqrt{5}}\right|\leq \frac{L_k}{\alpha^n}+ \frac{b^d|\beta^m|}{\alpha^n\sqrt{5} }+\frac{|\beta^n|}{\alpha^n}.
		\end{align*}
		  Since $|\beta|<1$, we get \begin{align*}
			\left|1-\frac{b^d}{\sqrt{5}\alpha^{n-m}}\right|&\leq \frac{2\alpha^k}{\alpha^n}+\frac{b^d}{\sqrt{5}\alpha^{n}}+\frac{1}{\alpha^{n}} \\
			&\leq \frac{2}{\alpha^{n-k}}+\frac{2b\alpha^k}{\sqrt{5}\alpha^{n}}+\frac{1}{\alpha^n}\\&<\frac{2}{\alpha^{n-k}}+\frac{20}{\sqrt{5}\alpha^{n-k}}+\frac{1}{\alpha^{n-k}}\\&<\frac{12}{\alpha^{n-k}}.
			\end{align*} 
		  Fix \begin{align}\label{gamma3}
			\Gamma_3:=\left|1-\frac{b^d}{\sqrt{5}\alpha^{n-m}}\right|<\frac{12}{\alpha^{n-k}}.
		\end{align}
		  Suppose $\Gamma_3=0,$ then $\sqrt{5}\alpha^{n-m}=b^d$. This is impossible because the right-hand side is rational while the left-hand side is irrational for $n>1500$. Thus $\Gamma_3\neq 0.$ Take \begin{align*}
			\gamma_1 &:= b, & \gamma_2 &:= \sqrt{5}, & \gamma_3 &:= \alpha, \notag\\
			b_1 &:= d, & b_2 &:= -1, & b_3 &:= -(n-m).
		\end{align*}
		  We know that $\gamma_1,\gamma_2$ and $\gamma_3$ are elements of $\mathbb{Q}(\sqrt{5})$ of degree $D:=2.$ Using the properties of height of algebraic numbers, we get 
		  \[h(\gamma_1)=h(b)=\log b, \quad h(\gamma_2)=h(\sqrt{5})=\dfrac{1}{2}\log 5,
		  \]
		   and 
		  \[
		  h(\gamma_3)=h(\alpha)=\dfrac{1}{2}\log \alpha.
		  \] 
		  Thus, we take \[A_1:=2\log b, \quad A_2:=\log 5 \quad \text{ and } A_3:=\log \alpha.\] Using the fact that $\max\{d,1,n-m\}<n-m+3,$ we take $B:=n-m+3.$ We proceed to apply Theorem \ref{log} with $t:=3$ and obtain
		  \begin{align*}
			\log {|\Gamma_3|}&>-1.4\cdot 30^6\cdot 3^{4.5}\cdot 2^2(1+\log 2)(1+\log{(n-m+3)})(2\log b)(\log 5)(\log \alpha), \\&>-1.51\cdot 10^{12}\log b (1+\log {(n-m+3)}).
		\end{align*}   Also from \eqref{gamma3}, we have that \begin{align*}
			\log {|\Gamma_3|} < \log {(4b)}-\log {\alpha^{n-k}}.
		\end{align*}   A comparison of these two inequalities yields 
		\begin{align}\label{T1}
			n-k<8.0\cdot 10^{12}\log b (1+\log {(n-m+3)}),
		\end{align}   
		which is valid for $2\leq b\leq 10.$  
	\end{proof}
	  Next we prove the following. 
	  \begin{lemma}
		All solutions to the Diophantine equation \eqref{main2} satisfy \[n<6.3\cdot 10^{30}(\log b)^4.\]  
		\end{lemma} 
		\begin{proof}
		Here, we examine the Diophantine equation \eqref{main2} in the form 
		\begin{align*}
			\alpha^n+\beta^n=b^dF_m+\alpha^k +\beta^k,
		\end{align*}   
		which leads to $
			\alpha^n-\alpha^k-b^dF_m=\beta^k-\beta^n$.  
		Thus \begin{align*}
			\alpha^n\left(1-\alpha^{k-n}\right)-b^dF_m=\beta^k-\beta^n.
		\end{align*}   Dividing through by $\alpha^n\left(1-\alpha^{k-n}\right)$ and taking absolutes on both sides gives \begin{align*}
			\left|1-\frac{b^dF_m}{\alpha^n\left(1-\alpha^{k-n}\right)}\right|&\leq \frac{1}{\alpha^n\left(1-\alpha^{k-n}\right)}\left(|\beta^k|+|\beta^n|\right)\\ &<\frac{2}{\alpha^n\left(1-\alpha^{-(n-k)}\right)} \\ 
			&\le \frac{2}{\alpha^n\left(1-\alpha^{-1}\right)}
			<\frac{6}{\alpha^n}.
		\end{align*} 
		  Fix \begin{align}\label{T3}
			\Gamma_4:=\left|1-\frac{b^dF_m}{\alpha^n\left(1-\alpha^{k-n}\right)}\right|<\frac{6}{\alpha^n},
		\end{align}   
		$\Gamma_4\neq 0$ via a similar argument for $\Gamma_2$. Thus we take the field containing the following $\gamma_i$'s as $\mathbb{Q}(\sqrt5)$ of degree $D=2$,
		\begin{align*}
		\gamma_1 &:= b, & \gamma_2 &:= \alpha, & \gamma_3 &:= \dfrac{F_m}{1-\alpha^{k-n}}, \notag\\
		b_1 &:= d, & b_2 &:= n, & b_3 &:= 1.
		\end{align*}
		The heights of the algebraic numbers $\gamma_1, \gamma_2$ and $\gamma_3$ are defined respectively by \[h(\gamma_1)=\log b, \quad h(\gamma_2)=\frac{1}{2}\log \alpha,\]   and 
		\begin{align*}			
			h(\gamma_3)&=h\left(\frac{F_m}{1-\alpha^{k-n}}\right)\leq h(F_m)+h(1-\alpha^{k-n}) \\ 
			&\leq \log {(F_m)}+\frac{|k-n|}{2}\log \alpha+\log 2 \\
			&< |m-1|\log \alpha + \frac{|k-n|}{2}\log \alpha +\log 2.
			\end{align*} 
		  From inequality \eqref{T2}, $m-1<n-k+4$, hence it follows that 
		  \[h(\gamma_3)<\frac{3(n-k)+8}{2}\log \alpha+\log 2.\] 
		  Thus we take \[A_1:=2\log b, \qquad A_2:=\log \alpha, \qquad A_3:=(3(n-k)+8)\log \alpha+2\log 2.\] 
		We take $B:=n+3>n-m+3>d$ and $t:=3.$ Proceeding via Theorem \ref{log}, we get		
		\begin{align*}
			\log {|\Gamma_4|}>-1.4\cdot 30^6 \cdot 3^{4.5}\cdot 2^2 (1+\log 2)(1+\log{(n+3)})(2\log b)(\log \alpha)((3(n-k)+8)\log \alpha+2\log 2).
		\end{align*}  Also from \eqref{T3}, $
			\log {|\Gamma_4|} < \log 6-n\log \alpha$. Thus, we get
			 \begin{align}\label{T5}
			n<1.9\cdot 10^{26}(\log n)^2(\log b)^2,
		\end{align}  
		where we have used the fact that
		 \begin{align*}
			1+\log {(n+3)}&=1+\log {n\left(1+\frac{3}{n}\right)}<1+\log {2n}
			<2\log n,
		\end{align*}  
		and
		\begin{align*}
			3(n-k)+8 &< 3(8.0\cdot 10^{12}\log b (1+\log {(n-m+3)}))+8, \\&<3(8.0\cdot 10^{12}\log b (1+\log {(n+3)}))+8, \\&<3(8.0\cdot 10^{12}\log b (2\log n)+8, \\&<4.8\cdot 10^{13}\log b \log n,
		\end{align*} 
		via Lemma \ref{lem3.3}. 
		To get an upper bound of $n$ in terms of $b$ from \eqref{T5}, we use Lemma \eqref{reduction} by defining 
		\[
		Z:=n, \qquad s:=2, \qquad \text{and}\qquad T:=1.9\cdot 10^{26}(\log b)^2.
		\]   
		Lemma \eqref{reduction} gives
		 \[
		 n<2^2(1.9\cdot 10^{26}(\log b)^2)(\log {(1.9\cdot 10^{26}(\log b)^2)})^2.
		 \]  
		 Therefore, we get
		  \begin{align*}
			n<6.3\cdot 10^{30}(\log b)^4.
	\end{align*}
	\end{proof}
	  We proceed to reduce the bounds obtained. Put \begin{align*}
		\Lambda_3:= d\log b -(n-m)\log \alpha +\log {\frac{1}{\sqrt{5}}}.
	\end{align*}   If $n-k>10,$ and $2\leq b\leq 10,$ we deduce from equation \eqref{gamma3} that \begin{align*}
		\Gamma_3:= \left|e^{\Lambda_3}-1\right|<\frac{12}{\alpha^{n-k}}<\frac{1}{2}.
	\end{align*}  
	It follows that 
	\begin{align*}
		\left|\Lambda_3\right|<\frac{18}{\alpha^{n-k}}.
	\end{align*}   
	Combining this with \eqref{gamma3}, we have that 
	\begin{equation}\label{bd}
		0<\left|d\frac{\log b}{\log \alpha}-(n-m)+\frac{\log{\left(\frac{1}{\sqrt{5}}\right)}}{\log \alpha}\right|<\frac{(18/\log \alpha)}{\alpha^{n-k}}.
	\end{equation}   
	Since $2\leq b \leq10$, we have that $n<6.3\cdot 10^{30}(\log b)^4<1.8\cdot 10^{32}$. Now, we apply Lemma \ref{duj} to \eqref{bd} with the data
	$	w:=n-k$, $A:=18/\log \alpha$,  $B:=\alpha$, $M:=1.8\cdot 10^{32}$ and
	 \[\tau:=\frac{\log b}{\log \alpha}, \qquad \mu:= \frac{\log {\left(\frac{1}{\sqrt{5}}\right)}}{\log \alpha}.\]   Let $q_i$ be the denominator of the $i$-th convergent of the continued fraction of $\tau.$ We applied Lemma \ref{duj} with SageMath and got the following results. 
	\begin{center}
		\begin{tabular}{|c|c|c|c|c|c|c|c|c|c|}
			\hline
			$b$ & $2$ & $3$ & $4$ & $5$ & $6$ & $7$ & $8$ & $9$ & $10$ \\ \hline
			$q_i$ & $q_{72}$ & $q_{64}$ & $q_{72}$ & $q_{69}$ & $q_{65}$ & $q_{74}$ & $q_{64}$ & $q_{60}$ & $q_{67}$ \\ \hline
			$\varepsilon>$ & $0.36$ & $0.33$ & $0.45$ & $0.42$ & $0.23$ & $0.34$ & $0.45$ & $0.47$ & $0.01$ \\ \hline
			$n-k<$ & $173$ & $171$ & $169$ & $171$ & $175$ & $171$ & $171$ & $171$ & $177$ \\ \hline
		\end{tabular}
	\end{center}
	In all cases, we can see that $n-k<177$ holds. Therefore we have $m\leq 182$ via \eqref{T2}.

	To proceed, we go back to \eqref{T3} and write
	  \[
	  \Lambda_4:=d\log b -n\log \alpha+\log \left(\frac{F_m}{\left(1-\alpha^{k-n}\right)}\right),
	  \]   
	  so that  
	  \[\left|\Gamma_4\right|:=\left|e^{\Lambda_4}-1\right|<\frac{6}{\alpha^n}.\]   Note that since $n > 1500$, we have $6/\alpha^n < 0.5$ which also implies that 
	  \begin{align}\label{sev}
		\left|\Lambda_4\right|<\frac{12}{\alpha^n}.
	\end{align}
	  Therefore, dividing both sides of \eqref{sev} by $\log \alpha$ which is positive, we get 
	  \begin{align*}
		0<\left|d\frac{\log b}{\log \alpha}-n+ \frac{\log { \left(\frac{F_m}{\left(1-\alpha^{k-n}\right)}\right)}}{\log \alpha}\right|<\frac{(12/\log \alpha)}{\alpha^n}.
	\end{align*}  
	Next, we take the following parameters 
	\[w:=n, \qquad A:=\frac{12}{\log \alpha}, \qquad B:=\alpha, \qquad M:=1.8\cdot 10^{32}>d,\]   
	and
	 \[
	 \tau:=\frac{\log b}{\log \alpha}, \qquad \mu:= \frac{\log \left(\frac{F_m}{\left(1-\alpha^{k-n}\right)}\right)}{\log \alpha},
	 \]   
	 so that we can use Lemma \ref{duj}. Let $q_i$ be the denominator of the $i$-th convergent of the continued fraction of $\tau$. Using SageMath, we applied Lemma \ref{duj} with the above parameters where $2 \leq b \leq10$, $0\leq m \leq 182$ and $10\leq n-k\leq 199$. We got the largest numerical bounds for $w:=n$ as shown below. \begin{center}
	\begin{tabular}{|c|c|c|c|c|c|c|c|c|c|}
		\hline
		$b$ & $2$ & $3$ & $4$ & $5$ & $6$ & $7$ & $8$ & $9$ & $10$ \\ \hline
		$q_i$ & $q_{77}$ & $q_{70}$ & $q_{77}$ & $q_{75}$ & $q_{75}$ & $q_{78}$ & $q_{72}$ & $q_{68}$ & $q_{77}$ \\ \hline
		$m>$ & $1$ & $1$ & $1$ & $1$ & $1$ & $1$ & $1$ & $1$ & $1$ \\ \hline
		$n-k$ & $200$ & $201$ & $201$ & $201$ & $200$ & $201$ & $201$ & $201$ & $201$ \\ \hline
		$\varepsilon>$ & $9\cdot 10^{-6}$ & $0.00001$ & $7\cdot 10^{-6}$ & $0.00001$ & $5\cdot 10^{-6}$ & $0.00001$ & $0.00002$ & $1.8\cdot 10^{-6}$ & $9.5\cdot 10^{-6}$ \\ \hline
		$n<$ & $210$ & $211$ & $210$ & $214$ & $208$ & $209$ & $211$ & $208$ & $209$ \\ \hline 
	\end{tabular}
	\end{center}
	Thus we can deduce that $n<214$ holds in all cases. This is a contradiction to our original assumption that $n>1500.$ This completes our proof of Theorem \ref{thm2}. \qed
	\section*{Conclusion}
	In this paper, we have completely determined all Lucas numbers that can be expressed as mixed base $b$ concatenations of Fibonacci and Lucas numbers for $2 \leq b \leq 10$, showing in each case that only finitely many such representations exist. These results extend earlier work on concatenations within single recurrence sequences to a mixed sequence and arbitrary base setting. A natural direction for future research is to investigate whether analogous finiteness results hold for mixed concatenations involving generalized linear recurrence sequences.
	\section*{Acknowledgments} 
	The first author thanks the Mathematics division of Stellenbosch university for funding his stay while working on this project. The second author gratefully acknowledges the financial support provided by the Erasmus+ Programme of the European Union for his mobility to the University of L'Aquila, Italy, where part of this work was completed.

	\section*{Address}
	$ ^{1} $ Mathematics Division, Stellenbosch University, Stellenbosch, South Africa. 
	
	Email: \url{hbatte91@gmail.com}
	\\
 	$ ^{2} $ Department of Mathematics,  Makerere University, Kampala, Uganda. 
	
	Email: \url{kaggwaprosper58@gmail.com}


\begin{thebibliography}{99}
		
		\bibitem{AT}
		Ad\'{e}dji, K.~N., Faye, M., \& Togb\'{e}, A. (2024).
		On the Diophantine equations $P_n=b^dQ_m+Q_k$ and $Q_n=b^dP_m+P_k$ involving Pell and Pell-Lucas numbers.
		\emph{Proceedings-Mathematical Sciences}, \textbf{134}(1), 14.
		
		\bibitem{KM}
		Ad\'{e}dji, K.~N., Bliznac Trebje\v{s}anin, M. (2024).
		On mixed $B$-concatenations of Pell and Pell--Lucas numbers which are Pell numbers.
		\emph{Mathematica Pannonica} \textbf{30}(1), 91--104.
		
		\bibitem{alan}
		Altassan, A., \& Alan, M. (2024).
		Fibonacci numbers as mixed concatenations of Fibonacci and Lucas numbers.
		\emph{Mathematica Slovaca}, \textbf{74}(3), 563--576.
		
		\bibitem{BW}
		Baker, A., \& W\"{u}stholz, G. (1993).
		Logarithmic forms and group varieties.
		\emph{Journal f\"{u}r die reine und angewandte Mathematik}, \textbf{442}, 19--62.
		
		\bibitem{banks}
		Banks, W.~D., \& Luca, F. (2005).
		Concatenations with binary recurrent sequences.
		\emph{Journal of Integer Sequences}, \textbf{8}(1), Art. 05.1.3, 19.
		
		\bibitem{bat}
		Batte, H. (2025).
		Lucas numbers that are palindromic concatenations of two distinct repdigits.
		\emph{Mathematica Pannonica}, \textbf{31}(1), 22--33.
		
		\bibitem{BravoLuca2014}
		Bravo, J.~J., \& Luca, F. (2014).
		Repdigits in k-Lucas sequences.
		\textit{Proceedings -- Mathematical Sciences}, \textbf{124}(2), 141--154.
		
		\bibitem{BugeaudMignotteSiksek2006}
		Bugeaud, Y., Mignotte, M., \& Siksek, S. (2006).
		Classical and modular approaches to exponential Diophantine equations I. 
		Fibonacci and Lucas perfect powers.
		\textit{Annals of Mathematics}, \textbf{163}(3), 969--1018.
		
		\bibitem{duj}
		Dujella, A., \& Peth\H{o}, A. (1998).
		A generalization of a theorem of Baker and Davenport.
		\emph{Quarterly Journal of Mathematics}, \textbf{49}(195), 291--306.
		
		\bibitem{guz}
		G{\'u}zman--Sanchez, S., \& Luca, F. (2014). Linear combinations of factorials and S-units in a binary recurrence sequence. {\it Annales Math{\'e}matiques du Qu{\'e}bec\/} {\bf  38}, 169--188.	
		
		\bibitem{matl}
		Matveev, E.~M. (2007).
		An explicit lower bound for a homogeneous rational linear form in logarithms of algebraic numbers II.
		\emph{Izvestiya: Mathematics}, \textbf{62}, 723--772.
		
		\bibitem{qu}
		Murty M. R., \& Esmonde, J. (2005). Problems in Algebraic Number Theory, second
		edition, \textit{Graduate Texts in Mathematics}, vol. 190, Springer-Verlag, New York.
		
		
		
	\end{thebibliography}
\end{document}